\documentclass[10pt,reqno]{amsart}
\usepackage{graphicx, amssymb, color,pdfsync}
\usepackage[all,cmtip]{xy}
\numberwithin{equation}{section}
\usepackage{mathrsfs}
\usepackage{tikz}
\usetikzlibrary{matrix,arrows}
\begin{document}
\newcommand{\s}{\vspace{0.2cm}}

\newtheorem{theo}{Theorem}
\newtheorem*{theo*}{Theorem}
\newtheorem{prop}{Proposition}
\newtheorem{coro}{Corollary}
\newtheorem{lemm}{Lemma}
\newtheorem{example}{Example}
\theoremstyle{remark}
\newtheorem{rema}{\bf Remark}
\newtheorem*{nota*}{\bf Notation}
\newtheorem{defi}{\bf Definition}

\title[On Jacobians with group action and coverings]{On Jacobians with group action and coverings}
\date{}
\author{Sebasti\'an Reyes-Carocca and Rub\'i E. Rodr\'iguez}
\address{Departamento de Matem\'atica y Estad\'istica, Universidad de La Frontera, Avenida Francisco Salazar 01145, Casilla 54-D, Temuco,  Chile.}
\email{sebastian.reyes@ufrontera.cl, rubi.rodriguez@ufrontera.cl}
\thanks{Partially supported by Postdoctoral Fondecyt Grant 3160002, Fondecyt Grant 1141099 and Anillo ACT 1415 PIA-CONICYT Grant}

\begin{abstract}

Let $S$ be a compact Riemann surface and let $H$ be a finite group. It is known that if $H$ acts on $S$ then there is a $H$-equivariant isogeny decomposition of the Jacobian variety $JS$ of $S,$ called the group algebra decomposition of $JS$ with respect to $H.$

If $S_1 \to S_2$ is a regular covering map, then it is also known that the group algebra decomposition of $JS_1$ induces an isogeny decomposition of $JS_2.$

In this article we deal with the converse situation. More precisely, we prove that the group algebra decomposition can be lifted under regular covering maps, under appropriate conditions.

\end{abstract}
\maketitle

\section{Introduction}Let $H$ be a finite group acting on a compact Riemann surface $S$. It is a known fact that this action induces an action of $H$ on the Jacobian variety $JS$ of $S$ and this, in turn, gives rise to a $H$-equivariant isogeny decomposition of $JS$ as a product of abelian subvarieties. This decomposition is called the {\it group algebra decomposition} of $JS$ with respect to $H$; see \cite{cr} and \cite{l-r}.

\s

The decomposition of Jacobians with group actions has been extensively studied in different settings, with applications to theta functions, to the theory of integrable systems and to the moduli spaces of principal bundles of curves, among others. The simplest case of such a decomposition is when $H$ is a group of order two; this fact was already noticed in 1895  by Wirtinger \cite{W} and used by Schottky-Jung in \cite{SJ}.

For decompositions of Jacobians with respect to other special groups, we refer to the articles \cite{ba}, \cite{d1}, \cite{CHQ}, \cite{nos}, \cite{hr}, \cite{leslie},  \cite{P}, \cite{PA}, \cite{d3}, \cite{d2} and \cite{d4}.

\s

Let $C$ be a compact Riemann surface admitting the action of a finite group $G.$ For every subgroup $N$ of $G$ consider the associated regular covering map $$C \to S=C_N$$given by the action of $N$ on $C.$ It was proved in \cite{cr} that the group algebra decomposition of $JC$ with respect to $G$ induces an isogeny decomposition of $JS.$ Furthermore, the factors arising in the {\it induced decomposition} of $JS$ are exactly the same as the ones arising in the decomposition of $JC,$ possibly with lower multiplicity.

\s

This article is mainly devoted to generalize the aforementioned result, by investigating the unstudied converse situation. More precisely, let $S$ be a compact Riemann surface admitting the action of a finite group $H,$ and consider any Riemann surface $C$ satisfying the next two conditions:

\begin{enumerate}
\item[(a)] there exists a regular covering map $\pi : C \to S$, and
\item[(b)] $C$ admits the action of a supergroup $G$ of the deck group $K$ of $\pi,$ in such a way that $K$ is normal in $G$ and $G/K \cong H.$
\end{enumerate}

The main result of this paper states that the group algebra decomposition of $JS$ with respect to $H$ {\it lifts} to obtain an isogeny decomposition of $JC,$ which is closely related to the group algebra decomposition of $JC$ with respect to $G.$  More precisely, the factors arising in the decomposition of $JC$ can be of two types, namely:
\begin{enumerate}
\item[(a)] the factors arising in the group algebra decomposition of $JS$, with exactly the same multiplicities, or
\item[(b)] factors whose product is isogenous to the Prym variety $P(C/S)$ associated to the covering map $\pi$.
\end{enumerate}
\s

As a matter of fact, the main result tells us about the compatibility between the group algebra decomposition and the classically known decomposition
$$JC \sim JS \times P(C/S)$$ associated to the covering map $\pi : C \to S$.

\s

We anticipate that in order to determine the group algebra decomposition of $JC$ with respect to $G$ it is necessary to have a complete knowledge of the rational irreducible representations of $G.$ By contrast, as we shall see later, the application of the main result of this paper significantly simplifies those computations, by reducing the problem to study the representations of $G$ that are not trivial in $K.$

As a further application of the main result, we derive a result related to decomposition of Jacobians $JC$ as products of Jacobians of quotients of $C.$

\s

This article is organized as follows. In Section \ref{pre} we shall briefly review the basic background; namely, group actions on Riemann surfaces, representation of groups, abelian varieties and the group algebra decomposition theorem for Jacobians. In Section \ref{lemmata} we shall prove two basic purely algebraic lemmata. The results of this paper will be stated and proved in Sections \ref{cuerpo} and \ref{kani}. Finally, in Section \ref{ejems} we shall exhibit two explicit examples in order to show how our results can be applied.

\s
{\bf Acknowledgments.} The authors are grateful to their colleague Angel Carocca for his helpful  suggestions throughout the preparation of this manuscript.

\section{Preliminaries} \label{pre}
\subsection{Group actions on Riemann surfaces}
Let $S$ be a compact Riemann surface and let $\mbox{Aut}(S)$ be its automorphism group. We recall that a finite group $H$ acts on $S$ if there is a monomorphism $\epsilon : H \to \mbox{Aut}(S).$ The space of orbits $S_H$ of the action of $H \cong \epsilon(H)$ on $S$ is endowed with a Riemann surface structure, such that the natural projection $\pi_H : S \to S_H$ is holomorphic. The degree of $\pi_H$ is the order $|H|$ of $H$ and the multiplicity of $\pi_H$ at $p \in S$ is $|H_p|,$ where $H_p$ denotes the stabilizer of $p$ in $H$. If $|H_p| \neq 1$ then $p$ is called a {\it branch point} of $\pi_H$; its image by $\pi_H$ is a {\it branch value} of $\pi_H$.

Let $\{p_1, \ldots, p_l\}$ be a maximal collection of non-$H$-equivalent branch points of $\pi_H$. The {\it signature} of the action of $H$ on $S$ is the tuple $(\gamma; m_1, \ldots, m_l)$ where $\gamma$ is the genus of the quotient $S_H$ and $m_i=|H_{p_i}|.$  The branch value $\pi_H(p_i)$ is said to be {\it marked} with $m_i.$ The Riemann-Hurwitz formula relates these numbers with the order of $H$ and the genus $g$ of $S.$ Namely, \begin{equation*} \label{R-R}
2g-2=|H| [   2\gamma - 2 +  \Sigma_{i=1}^l(1-\tfrac{1}{m_i}) ]
\end{equation*}

A tuple $(a_1, \ldots, a_{\gamma}, b_1, \ldots,
b_{\gamma}, c_1, \ldots,c_l)$ of elements of $H$ is called a
\textit{generating vector of $H$ of type $(\gamma; m_1, \ldots ,m_l)$} if the following conditions are satisfied:\begin{enumerate}
\item[(a)] $H$ is generated by $a_1, \ldots, a_{\gamma}, b_1, \ldots,
b_{\gamma}, c_1, \ldots,c_l;$
\item[(b)] the order of $c_i$ is $m_i$ for all $i,$ and
\item[(c)] $\Pi_{i=1}^{\gamma}(a_ib_ia_i^{-1}b_i^{-1}) \Pi_{i=1}^l c_i=1.$
\end{enumerate}

Riemann's existence theorem ensures that $H$ acts on a Riemann surface of genus $g$ with signature $(\gamma; m_1, \ldots, m_l)$ if and only if the Riemann-Hurwitz formula is satisfied and $H$ has a generating vector of type $(\gamma; m_1, \ldots ,m_l);$ see \cite{brou}.

We refer to \cite{Farkas} for further material related to Riemann surfaces and group actions.

\subsection{Representations of groups}  Let $H$ be a finite group and let $\rho : H \to \mbox{GL}(V)$ be a complex  representation of $H.$ Abusing notation, we shall also write $V$  to refer to the representation $\rho.$ The {\it degree} $d_V$ of $V$ is the dimension of $V$ as a complex vector space, and the {\it character} $\chi_V$ of $V$ is the map obtained by associating to each $h \in H$ the trace of the matrix $\rho(h).$ Two representations $V_1$ and $V_2$ are {\it equivalent} if and only if their characters agree; we write $V_1 \cong V_2.$ The {\it character field} $K_V$ of $V$ is the field obtained by extending the rational numbers by the values of the character of $V$.  The {\it Schur index} $s_V$ of $V$ is the smallest positive integer such that there exists a degree $s_V$ field extension $L_V$ of $K_V$  over which $V$ can be defined.

It is classically known that for each rational irreducible representation $W$ of $H$ there is a complex irreducible representation $V$ of $H$ such that \begin{equation} \label{llave}W \otimes_{\mathbb{Q}}\mathbb{C}  \cong (\oplus_{\sigma} V^{\sigma}) \oplus \stackrel{s_V}{\cdots} \oplus (\oplus_{\sigma} V^{\sigma})= s_V (\oplus_{\sigma} V^{\sigma})\end{equation}where the sum $\oplus_{\sigma} V^{\sigma}$ is taken over the Galois group associated to $\mathbb{Q} \le K_V.$ We shall say that $V$ is {\it associated} to $W.$ Two complex irreducible representations associated to the same rational irreducible representation are termed {\it Galois associated}.

Let $N$ be a subgroup of $H$ and consider$$V^N:=\{v \in V : \rho(h)(v)=v \,\, \mbox{for all} \,\, h \in N\},$$the vector subspace of $V$ consisting of those elements fixed under $N;$ let $d_V^N$ denote its dimension. By Frobenius Reciprocity theorem, $$d_V^N= \langle \rho_N , V \rangle_H$$where $\rho_N$ stands for the representation of $H$ induced by the trivial one of $N,$ and the brackets for the usual inner product of representations.

We refer to \cite{Serre} for further basic facts related to representations of groups.

\subsection{Complex tori and abelian varieties} A $g$-dimensional {\it complex torus} $X=V/\Lambda$ is the quotient of a $g$-dimensional complex vector space $V$ by a maximal rank discrete subgroup $\Lambda$. Each complex torus is an abelian group and a $g$-dimensional compact connected complex analytic manifold. {\it Homomorphisms} between  complex tori are holomorphic maps which are also group homomorphisms; we shall denote by $\mbox{End}(X)$ the ring of {\it endomorphisms} of $X,$ that is, homomorphisms of $X$ into itself. An {\it isogeny} is a surjective homomorphism with finite kernel;  isogenous tori are denoted by $X_1 \sim X_2.$ The isogenies of a complex torus $X$ into itself are  the invertible elements of the algebra of rational endomorphisms  $$\mbox{End}_{\mathbb{Q}}(X):=\mbox{End(X)} \otimes_{\mathbb{Z}} \mathbb{Q}.$$

An {\it abelian variety} is a complex torus which is also a complex projective algebraic variety. The Jacobian variety $JS$ of a compact Riemann surface $S$ is an (irreducible principally polarized) abelian variety; its dimension is the genus of $S.$ By the well-known Torelli's theorem, two compact Riemann surfaces are isomorphic if and only if their Jacobians are isomorphic (as principally polarized) abelian varieties.

Associated to every covering map $\pi : C \to S$ between compact Riemann surfaces there are two homomorphism; namely, the {\it pull-back} and the {\it norm} of $\pi:$ $$\pi^* : JS \to JC \hspace{0.5 cm} \mbox{and}\hspace{0.5 cm}N_{\pi}: JC \to JS.$$ As $\pi^*(JS)$ is an abelian subvariety of $JC$ isogenous to $JS,$ Poincar\'e's Reducibility theorem implies that there exists an abelian subvariety $P(C/S)$ of $JC$ such that  $$JC \sim JS \times P(C/S).$$

The factor $P(C/S)$ (which is termed the {\it Prym variety} associated to $\pi$) agrees with the connected component of zero of the kernel of the norm $N_{\pi}.$

We refer to \cite{bl} and \cite{debarre} for basic material on this topic.

\subsection{Group algebra decomposition theorem}
Let us suppose that $H$ is a finite group and that $W_1, \ldots, W_r$ are its rational irreducible representations.  It is classically known that each action $\epsilon_{H,S}: H \to \mbox{Aut}(S)$ of $H$ on $S$ induces a $\mathbb{Q}$-algebra homomorphism  $$\Phi_{H,S} : \mathbb{Q} [H]\to \text{End}_{\mathbb{Q}}(JS)$$where $\mathbb{Q}[H]$ stands for the rational group algebra of $H$ (see, for example \cite[p. 431]{bl}).

For every  $\alpha \in {\mathbb Q}[H]$ we define the abelian subvariety $$A_{\alpha} := {\rm Im} (\alpha)=\Phi_{H, S} (l\alpha)(JS) \subset JS$$where $l$ is some positive integer chosen such that $\Phi_{H,S}(l\alpha) \in \mbox{End}(JS)$.

Let $e_i$ be the idempotent of $\mathbb{Q}[H]$ given by
$$
e_i=\frac{d_{V_i}}{|H|} \sum_{h \in H} \mbox{tr}_{K_{V_i} |\mathbb{Q}} (\chi_{V_i}(h^{-1}))h,
$$
where $V_i$ is a complex irreducible representation of $H$ associated to $W_i$, of degree $d_{V_i}$, and $\mbox{tr}_{K_{V_i} |\mathbb{Q}}$ is the trace of the extension $\mathbb{Q} \le K_{V_i}.$ Then the equality $1 = e_1 + \cdots + e_r$ yields an isogeny $$JS \sim A_{e_1} \times \cdots \times A_{e_r}$$which is $H$-equivariant; the factors $A_i:=A_{e_i}$ above are called the {\it isotypical factors} of $JS$ with respect to $H$. See \cite{l-r}.

Additionally, there are $n_i=d_{V_i}/s_{V_i}$ idempotents $f_{i1},\dots, f_{in_i}$ in $\mathbb{Q}[H]$ such that $e_i=f_{i1}+\dots +f_{in_i},$ which provide $n_i$ pairwise isogenous abelian subvarieties of $JS.$ Let $B_i$ be one of them, for every $i.$ Thus $A_i \sim B_{i}^{n_i}$ and therefore the isogeny
\begin{equation} \label{lg}
JS \sim_H B_{1}^{n_1} \times \cdots \times B_{r}^{n_r}
\end{equation}is obtained. This last isogeny is called the {\it group algebra decomposition} of $JS$ with respect to $H;$ see \cite{cr}.

If the representations are labeled in such a way that $W_1(=V_1)$ denotes the trivial one (as we will do in this paper) then $n_1=1$ and $B_{1} \sim JS_H.$

\s

Let $N$ be a subgroup of $H$ and consider the associated regular covering map $S \to S_N.$ It was proved in \cite{cr} that the group algebra decomposition of $JS$ with respect to $H$ induces the following isogeny decomposition of $JS_N:$   \begin{equation*}
JS_N \sim  B_{1}^{{n}_1^N} \times \cdots \times B_{r}^{n_r^N} \,\,\, \mbox{ where } \,\,\, {n}_i^N=d_{V_i}^N/s_{V_i} \, .
\end{equation*}

The isogeny above provides a criterion to identify if a factor in the group algebra decomposition of $JS$ with respect to $H$ is isogenous to the Jacobian variety of a quotient of $S$ or isogenous to the Prym variety of an intermediate covering of $\pi_H$. More precisely, if two  subgroups $N \le N'$ of $H$ satisfy $$d_{V_i}^{N}- d_{V_i}^{N'}=s_{V_i}$$for some fixed $2 \le i \le r$ and $$d_{V_l}^{N} - d_{V_l}^{N'} = 0$$for all $l \neq i$
such that $\dim(B_{l}) \neq 0,$ then  \begin{equation*}
B_{i} \sim P(S_{N} / S_{N'}).
\end{equation*}Furthermore if, in addition, the genus of $S_{N'}$ is zero then $
B_{i} \sim JS_{N}.$ See also \cite{yo}.

\s

Assume that $(\gamma; m_1, \ldots, m_l)$ is the signature for the action of $H$ on $S$  and that the tuple $(a_1, \ldots, a_{\gamma}, b_1, \ldots,
b_{\gamma}, c_1, \ldots,c_l)$ is a generating vector representing this action. Following \cite[Theorem 5.12]{yoibero}, the dimension of $B_{i}$ in (\ref{lg}) is
\begin{equation}\label{dimensiones}
\dim (B_{i})=k_{V_i}\Big[d_{V_i}(\gamma -1)+\frac{1}{2}\Sigma_{k=1}^l (d_{V_i}-d_{V_i}^{\langle c_k \rangle} )\Big]  \end{equation}for $2 \le i \le r, $ where $k_{V_i}$ is the degree of the extension $\mathbb{Q} \le L_{V_i}.$

\section{Two basic algebraic lemmata} \label{lemmata}
Let $G$ be a finite group, let $K$ be a normal subgroup of $G$ and let $\Phi : G \to H$ be a surjective homomorphism of groups whose kernel is $K.$

\s

If $\rho: H \to \mbox{GL}(V)$ is a complex representation of $H$ then by precomposing $\rho$ by $\Phi$ we obtain a complex representation $$\tilde{\rho}:= \rho \circ \Phi : G \to \mbox{GL}(V)$$ of $G$ whose kernel contains $K.$ Conversely, given a complex representation $\tilde{\rho}: G \to \mbox{GL}(V)$ of $G$ which is trivial in $K,$ for $h \in H$ we define $$\rho(h):=\tilde{\rho}(g),$$where $g$ is chosen in such a way that $\Phi(g)=h.$ It is easy to see that $\rho$ is well-defined.

\begin{rema} \mbox{} \label{mouse} \it
\begin{enumerate}
\item[(a)] We shall write $V$ and $\tilde{V}$ instead of $\rho$ and $\tilde{\rho}$ respectively.
\item[(b)] Throughout the paper we shall use repeatedly the following obvious remark. After fixing a basis of the vector space $V$, the sets of matrices $$\{\rho(h): h \in H\} \hspace{0.5 cm} \mbox{and} \hspace{0.5 cm}\{\tilde{\rho}(g): g \in G\}$$agree, showing that the character fields, the degrees and the Schur indices of the representations $V$ and $\tilde{V}$ agree.
\end{enumerate}
\end{rema}

The next lemma is a particular case of \cite[Theorem 11.25]{Curtis}; we include a substantially simpler proof that fits  this context better.

\begin{lemm} \label{cargador}
Let $G$ be a finite group and let $K$ be a normal subgroup of $G.$  If $H=G/K$ then the correspondence $V \mapsto \tilde{V}$ defines a bijection between
\begin{enumerate}
\item[(a)] the set of complex irreducible representations of $H$, and
\item[(b)] the set of complex irreducible representations of $G$ that are trivial in $K.$
\end{enumerate}
\end{lemm}

\begin{proof}
We only have to prove that the rule $V \mapsto \tilde{V}$ restricts to a bijection at the level of complex representations that are irreducible. To do this, we recall the well-known fact that the irreducibility of a representation $U$ of a group $\mathcal{G}$ is equivalent to $\langle U, U\rangle_{\mathcal{G}}=1.$ Now, the proof is clear after noticing that $$\langle V, V \rangle_H=\langle \tilde{V}, \tilde{V} \rangle_G$$for every complex representation $V$ of $H.$
\end{proof}

\begin{lemm} \label{azul}
Let $G$ be a finite group and let $K$ be a normal subgroup of $G.$  If $H=G/K$ then the correspondence  $V \mapsto \tilde{V}$ defines a bijection between
\begin{enumerate}
\item[(a)] the set of rational irreducible representations of $H$, and
\item[(b)] the set of rational irreducible representations of $G$ that are trivial in $K.$
\end{enumerate}
\end{lemm}

\begin{proof}
If $V_1, V_2$ are representations of $H,$ then it is straightforward to check that $$ \widetilde{V_1 \oplus V_2}=\tilde{V}_1 \oplus \tilde{V}_2 \hspace{0.5 cm} \mbox{and} \hspace{0.5 cm}      \widetilde{(V_1^{\sigma})}=(\tilde{V}_1)^{\sigma}$$for every $\sigma$ in the Galois group associated to the extension $\mathbb{Q} \le K_{V_1}=K_{\tilde{V}_1}.$ The result follows directly from Lemma \ref{cargador}, from Remark \ref{mouse} and from the way in which the rational irreducible representations of a group are constructed (see \eqref{llave}).
\end{proof}

\section{Stamement and proof of the results} \label{cuerpo}

Let $S$ be a compact Riemann surface and let $H$ be a finite group acting on $S.$ In this section we assume the existence of a compact Riemann surface $C$ such that:

\begin{enumerate}
\item[(a)] there exists a regular covering map $\pi : C \to S$, and that
\item[(b)] $C$ admits the action of a supergroup $G$ of the deck group $K$ of $\pi,$ in such a way that $K$ is normal in $G$ and $G/K \cong H.$
\end{enumerate}

 Note that $$S_H \cong (C/K)/(G/K) \cong 	C_G $$as Riemann surfaces, and the following diagram commutes.$$\begin{tikzpicture}[node distance=2 cm, auto]
  \node (P) {$C$};
  \node (Q) [right of=P, node distance=2.5 cm] {$S$};
  \node (A) [below of=P, node distance=1 cm] {$C_G$};
  \node (B) [below of=Q, node distance=1 cm] {$S_H$};
  \draw[->] (P) to node  {$\pi$} (Q);
  \draw[->] (A) to node {$\cong$} (B);
  \draw[->] (P) to node [swap] {$\pi_G$} (A);
  \draw[->] (Q) to node  {$\pi_H$} (B);
\end{tikzpicture}$$

\begin{prop} \label{vgvg} Let $\Phi : G \to H$ be a surjective homomorphism of groups whose kernel is $K.$ If $$\sigma=(a_1, \ldots, a_{\gamma}, b_1, \ldots,
b_{\gamma}, c_1, \ldots,c_l)$$is a generating vector of type $(\gamma; m_1, \ldots, m_l)$ representing the action of $G$ on $C,$ then the tuple$$\Phi(\sigma)=(\Phi(a_1), \ldots, \Phi(a_{\gamma}), \Phi(b_1), \ldots,
\Phi(b_{\gamma}), \Phi(c_1), \ldots,\Phi(c_l))$$is a generating vector of type $$(\gamma; \tfrac{m_1}{d_1}, \ldots, \tfrac{m_l}{d_l} )$$and represents the action of $H$ on $S,$ where $d_j=|\langle c_j \rangle \cap K|.$
\end{prop}

\begin{proof} Let $p_j$ denote the branch value of $\pi_G$ marked with $m_j.$ Then, over $p_j$ by $\pi_G$ there are exactly $t_j:=|G|/m_j$ different points; say $$\pi_G^{-1}(p_j)=\{q_{j,1}, \ldots, q_{j,t_j}\}.$$ Following \cite[Section 3.1]{yoibero}, the multiplicity $d_{j,i}$ of $\pi$ at $q_{j,i}$ agrees with the order of the intersection group $$g_i \langle c_j \rangle g_i^{-1} \cap K$$for a suitable representative $g_i$ of the left cosets of the normalizer of $\langle c_j \rangle$ in $G.$

The normality of $K$ in $G$ implies that $$d_{j,i} :=|g_i(\langle c_j \rangle \cap K)g_i^{-1}|=|\langle c_j \rangle \cap K|$$
is independent of $i$, so we can set $d_j=d_{j,i}$ ; note that $d_j$ divides $m_j$.

It follows that $\pi_G^{-1}(p_j)$ is sent by $\pi$ to $$(|G|/m_j)/(|K|/d_j)=|H|/(m_j/d_j)$$different points in $S.$ This, in turn, says that $\pi_H$ has a branch value marked with $m_j/d_j$ and, in the end, that the signature of the action of $H$ on $S$ is $$(\gamma; \tfrac{m_1}{d_1}, \ldots, \tfrac{m_l}{d_l} ).$$

Now, the proof is done after noticing that the elements of the tuple $\Phi(\sigma)$ generate $H,$ that the order of $\Phi(c_j)$ is $m_j/d_j,$ and that $\Phi(1)=1.$
\end{proof}

Let us employ the following notation:$$\mbox{End}_{\mathbb{Q}, G}(JC):=\Phi_{G,C} (\mathbb{Q}[G]) \subset \mbox{End}_{\mathbb{Q}}(JC)$$ and $$\mbox{End}_{\mathbb{Q}, H}(JS):=\Phi_{H,S} (\mathbb{Q}[H]) \subset \mbox{End}_{\mathbb{Q}}(JS).$$

\begin{prop} \label{azucar}
Let $\Phi : G \to H$ be a surjective homomorphism of groups whose kernel is $K.$ Then $\Phi$ induces surjective $\mathbb{Q}$-algebra homomorphisms
$$\hat{\Phi} :  \mathbb{Q}[G] \to \mathbb{Q}[H] \hspace{0.5 cm} \mbox{and} \hspace{0.5 cm} \check{\Phi} :\mbox{End}_{\mathbb{Q}, G}(JC) \to \mbox{End}_{\mathbb{Q}, H}(JS),$$
that make the following diagram commutative:$$\begin{tikzpicture}[node distance=2 cm, auto]
  \node (P) {$\mathbb{Q}[G]$};
  \node (Q) [right of=P, node distance=3 cm] {$\mbox{End}_{\mathbb{Q}, G}(JC)$};
  \node (A) [below of=P, node distance=1.3 cm] {$\mathbb{Q}[H]$};
  \node (B) [below of=Q, node distance=1.3 cm] {$\mbox{End}_{\mathbb{Q}, H}(JS)$};
  \draw[->] (P) to node {$\Phi_{G,C}$} (Q);
  \draw[->] (A) to node {$\Phi_{H,S}$} (B);
  \draw[->] (P) to node [swap] {$\hat{\Phi}$} (A);
  \draw[->] (Q) to node  {$\check{\Phi}$} (B);
\end{tikzpicture}$$
\end{prop}

\begin{proof} The map $\hat{\Phi}$ is the natural extension of $\Phi$ by linearity; namely $$ \mathbb{Q}[G] \ni \sum_{g \in G} \lambda_g g \mapsto \sum_{g \in G} \lambda_g \Phi(g)  =\sum_{h \in H} \mu_h h \in \mathbb{Q}[H]$$where \begin{equation} \label{nu}\mu_h=\sum_{g \in G, \Phi(g)=h} \lambda_g.\end{equation}

It is not a difficult task to check that this is, in fact, a surjective homomorphism between $\mathbb{Q}$-algebras.

Since a typical element of $\mbox{End}_{\mathbb{Q}, G}(JC)$ is of the form $\Sigma_{g \in G} \lambda_g \Phi_{G,C}(g),$ in a similar way as before, we define $\check{\Phi}$ by the rule $$\mbox{End}_{\mathbb{Q}, G}(JC) \ni \sum_{g \in G} \lambda_g \Phi_{G,C}(g) \mapsto \sum_{g \in G} \lambda_g \Phi_{H,S}(\Phi(g)),$$and the last expresion can be rewritten as $$\sum_{h \in H} \mu_h \Phi_{H,S}(h)  \in \mbox{End}_{\mathbb{Q}, H}(JS)$$with $\mu_h$ as defined in \eqref{nu}. Again, it is not difficult to check that $\check{\Phi}$ is a surjective homomorphism between $\mathbb{Q}$-algebras, and the diagram commutes by construction.
\end{proof}

\s

We recall that associated to the covering map $\pi: C \to S$ there is a surjective homomorphism between the corresponding Jacobians: $$N_{\pi}: JC \to JS,$$called the {\it norm} of $\pi.$ This map corresponds to the {\it push-forward} with respect to $\pi,$ when we regard the Jacobian variety as the group of (equivalence classes of) degree zero divisors (or line bundles) on the Riemann surface; see e.g. \cite[Chapter 11]{bl}.

\begin{prop} \label{normco}
For every $\varphi \in \mbox{End}_{\mathbb{Q}, G}(JC)$ the following diagram commutes:$$\begin{tikzpicture}[node distance=2 cm, auto]
  \node (P) {$JC$};
  \node (Q) [right of=P, node distance=3 cm] {$JC$};
  \node (A) [below of=P, node distance=1.3 cm] {$JS$};
  \node (B) [below of=Q, node distance=1.3 cm] {$JS$};
  \draw[->] (P) to node {$\varphi$} (Q);
  \draw[->] (A) to node {$\check{\Phi}(\varphi)$} (B);
  \draw[->] (P) to node [swap] {$N_{\pi}$} (A);
  \draw[->] (Q) to node  {$N_{\pi}$} (B);
\end{tikzpicture}$$
\end{prop}

\begin{proof}
Let $p_1, \ldots, p_t$ be different points of $C$ and consider the degree zero divisor $$D=\sum_{i=1}^t n_ip_i$$as a point of $JC.$ If $\varphi=\sum_{g \in G} \lambda_g \Phi_{G,C}(g) \in \mbox{End}_{\mathbb{Q}, G}(JC)$ then $\varphi(D)$ is defined as $$\varphi(D)=\sum_{g \in G} \lambda_g \sum_{i=1}^t n_i \epsilon_{G,C}(g)(p_i), $$where $\epsilon_{G,C} : G \to \mbox{Aut}(C)$ is the monomorphism defining the action of $G$ on $C.$

Thus \begin{equation} \label{guita} (N_{\pi} \circ \varphi )(D)= \sum_{g \in G} \lambda_g \sum_{i=1}^t  n_i  \pi \circ \epsilon_{G,C}(g)(p_i).\end{equation}

Now, as $K$ is a normal subgroup of $G,$ we have that $$\epsilon_{H,S}(h) \circ \pi = \pi \circ \epsilon_{G,C}(g)$$where $\epsilon_{H,S}: H \to \mbox{Aut}(S)$  is the monomorphism defining the action of $H$ on $S,$ and $h=\Phi(g).$ It follows that the equality \eqref{guita} can be rewritten as $$(N_{\pi} \circ \varphi) (D)= \sum_{h \in H} \mu_h \sum_{i=1}^t  n_i   \epsilon_{H,S}(h) \circ \pi (p_i)$$where $\mu_h$ is as in \eqref{nu}. The last expression equals  $$\sum_{h \in H} \mu_h \epsilon_{H,S}(h) \, \pi \sum_{i=1}^t  n_ip_i = ( \check{\Phi}(\varphi) \circ N_{\pi})(D)$$and the proof follows.
\end{proof}

\s

We are now in position to state and prove the main result of this paper:

\s

\begin{theo} \label{theolindo}
Let $S$ be a compact Riemann surface with action of a finite group $H,$ and let \begin{equation} \label{e1}JS \sim_H JS_H \times B_2^{n_2} \times \cdots \times B_r^{n_r}\end{equation}be the group algebra decomposition of $JS$ with respect to  $H.$

\s

Assume the existence of a compact Riemann surface $C$ such that: \begin{enumerate}
\item[(a)] there exists a regular covering map $\pi : C \to S$, and
\item[(b)] $C$ admits the action of a supergroup $G$ of the deck group $K$ of $\pi,$ in such a way that $K$ is normal in $G$ and $G/K \cong H.$
\end{enumerate}Then the group algebra decomposition of $JC$ with respect to $G$ is $$JC \sim_G (JS_H \times {B}_2^{n_2} \times \cdots \times {B}_r^{n_r}) \times (\tilde{B}_{r+1}^{\tilde{n}_{r+1}} \times \cdots \times \tilde{B}_s^{\tilde{n}_s})$$where the factors $\tilde{B}_i,$ for $r+1 \le i \le s,$ are associated to the rational irreducible representations of $G$ that are not trivial in $K.$
\end{theo}

\begin{proof} Let $W_1 \ldots, W_r$ be the rational irreducible representations of $H$. By Lemma \ref{azul}, the representations $\tilde{W}_1, \ldots, \tilde{W}_r$ of $G$ are precisely the rational irreducible representations of $G$ that are  trivial in $K$.  If we denote by $\tilde{W}_{r+1}, \ldots, \tilde{W_s}$ the rational irreducible representations of $G$ that are not trivial in $K,$ then the group algebra decomposition of $JC$ with respect to $G$ is \begin{equation} \label{blanco}JC \sim_G (JC_G \times \tilde{B}_2^{\tilde{n}_2} \times \cdots \times \tilde{B}_r^{\tilde{n}_r}) \times ( \tilde{B}_{r+1}^{\tilde{n}_{r+1}} \times \ldots \times  \tilde{B}_s^{\tilde{n}_s}),\end{equation}where the factor $\tilde{B}_i$ is associated to $\tilde{W}_i$ and $\tilde{n}_i=d_{\tilde{V}_i}/s_{\tilde{V}_i}.$

As $S_H \cong C_G$ and as $\tilde{n}_i =n_i$ for $2 \le i \le r$ (see Remark \ref{mouse}), the isogeny \eqref{blanco} can be rewritten as \begin{equation} \label{piojo}JC \sim_G (JS_H \times \tilde{B}_2^{{n}_2} \times \cdots \times \tilde{B}_r^{{n}_r}) \times ( \tilde{B}_{r+1}^{\tilde{n}_{r+1}} \times \ldots \times  \tilde{B}_s^{\tilde{n}_s}),\end{equation}and therefore to prove the theorem we have to show that $\tilde{B}_i$ and $B_i$ are isogenous, for each $2 \le i \le r.$ To accomplish this task, we proceed in three steps.

\s

{\bf Claim 1.} The dimensions of $B_i$ and $\tilde{B}_i$ agree for every $2 \le i \le r.$

\s

Let us suppose the signature of the action of $G$ on $C$ to be $(\gamma; m_1, \ldots, m_l)$ and  $$\sigma=(a_1, \ldots, a_{\gamma}, b_1, \ldots,
b_{\gamma}, c_1, \ldots,c_l)$$to be a generating vector representing this action. Note that for every $2 \le i \le r$ and $1 \le k \le l$ we have $$ \tilde{\rho_i}(c_k)=(\rho_i \circ \Phi)(c_k)= \rho_i (\Phi(c_k))$$ and, consequently, the subspace of $V_i$ fixed under the subgroup $\langle c_k \rangle$ of  $G$ $$\tilde{V}_i^{\langle c_k \rangle}=\{v \in V_i : \tilde{\rho_i}(c_k)(v)=v\}$$agrees with the subspace of $V_i$ fixed under the subgroup $\langle \Phi(c_k) \rangle$ of $H$ $$V_i^{\langle \Phi(c_k) \rangle}=\{v \in V_i : {\rho_i}(\Phi(c_k))(v)=v\}.$$

By Proposition \ref{vgvg},  the signature of the action of $H$ on $S$ is $$(\gamma; \tfrac{m_1}{d_1}, \ldots, \tfrac{m_l}{d_l})$$ and the tuple $\Phi(\sigma)$ is a generating vector representing this action. Now, we employ the aforementioned generating vectors in order to apply the formula (\ref{dimensiones}) to \eqref{e1} and \eqref{blanco}. It follows that\begin{multline*}
\dim (B_{i}) = k_{V_i}\Big[d_{V_i}(\gamma -1)+\frac{1}{2}\sum_{k=1}^l (d_{V_i}-d_{V_i}^{\langle \Phi(c_k))})\Big]  \\
 =k_{\tilde{V}_i}\Big[d_{\tilde{V}_i}(\gamma -1)+\frac{1}{2}\sum_{k=1}^l (d_{\tilde{V}_i}-d_{\tilde{V}_i}^{\langle c_k)})\Big] =\dim(\tilde{B}_i)
  \end{multline*}and the proof of Claim 1 is done.

\s
\s

Let us now consider the isotypical factors $\tilde{A}_i \sim \tilde{B}_i^{n_i}$ and ${A}_i \sim {B}_i^{n_i}$ of $JC$ and $JS$ with respect to $G$ and $H$ respectively, for each $2 \le i \le r.$

We recall that $\tilde{A}_i$ and ${A}_i$ agree with the image of the endomorphisms $\Phi_{G,C}(\tilde{e}_i)$ and $\Phi_{H,S}(e_i)$ respectively, where $$ \tilde{e}_i=\frac{d_{\tilde{V}_i}}{|G|} \sum_{g \in G} \mbox{tr}_{K_{\tilde{V}_i} |\mathbb{Q}} (\chi_{\tilde{V}_i}(g^{-1}))g\hspace{ 0.4 cm} 	\mbox{and} \hspace{ 0.4 cm}    e_i=\frac{d_{V_i}}{|H|} \sum_{h \in H} \mbox{tr}_{K_{V_i} |\mathbb{Q}} (\chi_{V_i}(h^{-1}))h. $$

\s

{\bf Claim 2.} $\hat{\Phi}(\tilde{e}_i)=e_i$ for every $2 \le i \le r.$

\s

The next equality follows directly from the definition of $\hat{\Phi}:$
$$\hat{\Phi}(\tilde{e}_i) = \frac{d_{\tilde{V}_i}}{|G|} \sum_{g \in G} \mbox{tr}_{K_{\tilde{V}_i} |\mathbb{Q}}(\chi_{\tilde{V}_i}(g^{-1})){\Phi}(g).$$

Furthermore, the fact that $\tilde{V}_i$ is trivial in $K$ together with Remark \ref{mouse} imply$$\sum_{g \in G, \, \Phi(g)=h_0} \mbox{tr}_{K_{\tilde{V}_i} |\mathbb{Q}}(\chi_{\tilde{V}_i}(g^{-1}))=|K| \mbox{tr}_{{K}_{V_i} |\mathbb{Q}}(\chi_{{V}_i}(h_0^{-1})$$for each $h_0 \in H.$ Thereby

$$\hat{\Phi}(\tilde{e}_i)=\frac{d_{{V}_i}}{|G|} \sum_{h \in H} |K| \mbox{tr}_{K_{V_i} |\mathbb{Q}}(\chi_{{V}_i}(h^{-1}))h =e_i.$$

\s

{\bf Claim 3.} The norm $N_{\pi}$ restricts to a homomorphism $\tilde{A}_i \to A_i$ for each $2 \le i \le r.$

\s

Propositions \ref{azucar} and \ref{normco} together with Claim 2 imply that the equality $$\Phi_{H,S}({e}_i) \circ N_{\pi}= N_{\pi} \circ \Phi_{G,C}(\tilde{e}_i)$$holds, for each $2 \le i \le r.$ By considering images, we obtain that $$A_i=N_{\pi}(\tilde{A}_i)$$proving Claim 3.

\s

Finally, from Claim 1 follows the fact that the factors $\tilde{A}_i$ and $A_i$ have the same dimension, showing that the homomorphism $N_{\pi}: \tilde{A}_i \to A_i$ is, in fact, an isogeny; consequently $\tilde{B}_i$ and $B_i$ are isogenous for every $2 \le i \le r.$

This brings the proof to an end.
\end{proof}

\begin{rema} \it As $S$ is isomorphic to the quotient $C_K,$ the isogeny decomposition \eqref{piojo} yields the following isogeny:
\begin{eqnarray}
JS &\sim&  (JS_H \times \tilde{B}_2^{{n}_2^K} \times \cdots \times \tilde{B}_r^{{n}_r^K}) \times ( \tilde{B}_{r+1}^{\tilde{n}_{r+1}^K} \times \cdots \times  \tilde{B}_s^{\tilde{n}_s^K}) \\
 \, &=& (JS_H \times \tilde{B}_2^{{n}_2} \times \cdots \times \tilde{B}_r^{{n}_r}) \times ( \tilde{B}_{r+1}^{\tilde{n}_{r+1}^K} \times \cdots \times  \tilde{B}_s^{\tilde{n}_s^K}).  \label{taza}\end{eqnarray}

Note that from the comparison of dimensions in the isogenies \eqref{e1} and \eqref{taza} it follows that the product $$\tilde{B}_{r+1}^{\tilde{n}_{r+1}^K} \times \cdots \times  \tilde{B}_s^{\tilde{n}_s^K}$$must be zero.

The previous observation  says that, if $\tilde{V}$ is a representation of $G$ which is not trivial in $K$ then $\tilde{V}$ is ``totally non-trivial". More precisely: \begin{equation} \label{todon} n_{\tilde{V}}^K \neq n_{\tilde{V}} \implies  n_{\tilde{V}}^K=0.\end{equation}

The implication above can be understood in purely algebraic terms, as follows.

\s

Let  $\tilde{V}$ be a complex irreducible representation to $G.$ Following \cite[p. 264]{Suzuki}, as $K$ is a normal subgroup of $G,$ the restriction $\tilde{V}|_{K}$ of $\tilde{V}$ to $K$ decomposes in terms of irreducible representations of $K$ as \begin{equation} \label{ddecc} m (\rho_1 \oplus \cdots \oplus \rho_{\nu}),\end{equation}where $m$ is some non-negative integer and $\rho_1, \ldots, \rho_{\nu}$ are conjugate representations of $\rho_1.$

If we assume that $n_{\tilde{V}}^K \neq 0$ then the trivial representation $\chi_{0,K}$ of $K$ must appear in the decomposition \eqref{ddecc}. Without loss of generality, we can suppose $ \rho_1 = \chi_{0,K}$ and therefore, as $\rho_2, \ldots \rho_s$ are conjugate representations of $\rho_1$, we obtain that $$ \tilde{V}|_{K} \cong (m \nu) \chi_{0,K}.$$

Note that $m \nu$ equals the degree $d_{\tilde{V}}$ of $\tilde{V}$ and, by Frobenius Reciprocity theorem, we obtain that  $$d_{\tilde{V}}^K=\langle \tilde{V}, \rho_K \rangle_G=\langle \tilde{V}|_{K}, \chi_{0,K} \rangle_K=\langle d_{\tilde{V}} \chi_{0,K} , \chi_{0,K} \rangle_K= d_{\tilde{V}}$$showing that $n_{\tilde{V}}^K = n_{\tilde{V}},$ as desired.

\s

In terms of $\mathbb{Q}$-algebras, the implication \eqref{todon} reveals the existence of a surjective $\mathbb{Q}$-algebra homomorphism $$a: \mathbb{Q}[G]\cong \mathbb{Q}[G]^K \oplus R \to \mathbb{Q}[H \cong G/K]$$such that $a(R)=0,$ and $a|_{\mathbb{Q}[G]^K}$ is an isomorphism onto $\mathbb{Q}[H].$ More precisely, if $$G=g_1 K \cup \ldots \cup g_dK$$ is a decomposition of $G$ into left cosets of $K$, with $d=|H|,$ then $a$ is defined as: $$e_{\tilde{W}} \mapsto \frac{1}{|G|}\sum_{g \in G} \tilde{\alpha}_g \Phi (g)=\frac{1}{|G|}\sum_{i=1}^d \tilde{\alpha}_{g_i} \Phi(g_i) \sum_{k \in K} \tilde{\alpha}_k,$$with $\tilde{\alpha}_g=d_{\tilde{V}}\mbox{tr}_{K_{\tilde{V}} |\mathbb{Q}}(\chi_{\tilde{V}}(g^{-1})).$

Note that $R$ is generated by those elements $e_{\tilde{W}} \in \mathbb{Q}[G]$ with $\tilde{W}$ being not trivial in $K.$ In this case the result \eqref{todon} ensures that $$n_{\tilde{V}}^K=\frac{1}{|K|}\sum_{k \in K} \chi_{\tilde{V}}(k)=0.$$ In particular $$ \sum_{k \in K} \tilde{\alpha}_k=d_{\tilde{V}} \sum_{k \in K}\mbox{tr}_{K_{\tilde{V}} |\mathbb{Q}}(\chi_{\tilde{V}}(k^{-1}))=0$$ showing that $a(R)=0.$

By contrast, as $\mathbb{Q}[G]^K$ is generated by those elements $e_{\tilde{W}} \in \mathbb{Q}[G]$ with $\tilde{W}$ being trivial in $K,$ it is easy to see that the restriction of $a$ to $\mathbb{Q}[G]^K$ is defined by $$e_{\tilde{W}} \to e_{{W}}$$showing that $a$ restricts to an isomorphism from $\mathbb{Q}[G]^K$ onto $\mathbb{Q}[H].$
\end{rema}

\s

According to the notation used in the proof of Theorem \ref{theolindo}, if $B_i$ is the abelian subvariety of $JS$ associated to the representation $W_i$ of $H$ then $\tilde{B}_i$ denotes the corresponding  abelian subvariety of $JC$ associated to the representation $\tilde{W}_i$ of $G.$

We have proved that $B_i$ and $\tilde{B}_i$ are isogenous. Moreover:

\begin{theo} \label{corolindo} Consider subgroups $N$ and $N'$ of $H$ such that $N \le N'$ and $i \in \{2, \ldots, r\}$. Then:
\begin{enumerate}
\item $$B_i \sim JS_N  \iff \tilde{B}_i \sim JC_{\Phi^{-1}(N)}.$$
\item $$B_i \sim P(S_N / S_{N'})  \iff  \tilde{B}_i \sim P(C_{\Phi^{-1}(N)} / C_{\Phi^{-1}(N')})  .$$
\end{enumerate}
\end{theo}

\begin{proof} By using the same arguments employed in the proof of Proposition \ref{vgvg}, it is not difficult to see that if $\sigma$ is a generating vector representing the action of $\Phi^{-1}(N)$ on $C$ then $\Phi(\sigma)$ is a generating vector representing the action of $N$ on $S.$ Thus, in a similar way as done in the proof of Theorem \ref{theolindo}, it can be seen that \begin{equation} \label{key} d_{\tilde{V}_j}^{\Phi^{-1}(N)}= d_{{V}_j}^N \hspace{0.4 cm}     \end{equation}for every $2 \le j \le r.$
Let us assume that $\tilde{B}_i \sim JC_{\Phi^{-1}(N)}.$ This is equivalent to the genus of $C_G$ being zero and  \begin{displaymath}
d_{\tilde{V}_j}^{\Phi^{-1}(N)}= \left\{ \begin{array}{ll}
 \,\,0 & \textrm{if $j \neq i$}\\
s_{\tilde{V}_i} & \textrm{if $j=i$}
  \end{array} \right.
\end{displaymath}for all $2 \le j \le r$ such that $\tilde{B}_j \neq 0.$ By \eqref{key} it is clear that  $B_i \sim JS_N.$

The converse is similar. Let us assume that $B_i \sim JS_N.$ This is equivalent to the genus of $S_H$ being zero and
\begin{displaymath}
d_{{V}_j}^N= \left\{ \begin{array}{ll}
 \,\,0 & \textrm{if $j \neq i$}\\
s_{V_i} & \textrm{if $j=i$}
  \end{array} \right.
\end{displaymath}for all $2 \le j \le r$ such that $B_j \neq 0.$  Now, \eqref{key} implies that $$JC_{\Phi^{-1}(N)} \sim \tilde{B}_i \times P_1$$for some abelian subvariety $P_1$ of $JC.$

As the dimensions of $B_i$ and $\tilde{B}_i$ agree, to prove part $(1)$ it is enough to check that the genera of $S_N$ and of $C_{\Phi^{-1}(N)}$ agree. This fact follows from $$S_N \cong (C/K)(\Phi^{-1}(N)/K) \cong C_{\Phi^{-1}(N)}.$$

\s

To prove part $(2)$ we proceed analogously. The equality \begin{equation*} \label{key2} d_{\tilde{V}_j}^{\Phi^{-1}(N)}-d_{\tilde{V}_j}^{\Phi^{-1}(N')}= d_{{V}_j}^N-d_{{V}_j}^{N'}\end{equation*}holds for every  $2 \le j \le r$ and the sufficient condition is clearly satisfied. Conversely, if $B_i \sim P(S_N /S_{N'})$ then$$P(C_{\Phi^{-1}(N)} / C_{\Phi^{-1}(N')}) \sim \tilde{B}_i \times P_2$$for some abelian subvariety $P_2$ of $JC.$ Now, as $B_i$ and $\tilde{B}_i$ have the same dimension, the result follows directly after noticing that the dimensions of $P(S_N / S_{N'})$ and of $P(C_{\Phi^{-1}(N)} /  C_{\Phi^{-1}(N')})$ agree.
\end{proof}

\section{Jacobian isogenous to product of Jacobians}\label{kani} By arguing as in the proof of Theorem \ref{theolindo}, we are able to derive conditions under which the Jacobian of a Riemann surface $C$ with group action can be decomposed as a product of Jacobians of quotients of $C$ by subgroups.

\begin{lemm}
Let $\tilde{W}$ be a rational irreducible representation of $G$ and let $\tilde{B}_{\tilde{W}}$ be the factor associated to it in the group algebra decomposition of $C$ with respect to $G.$

Assume the existence of two subgroups $K_1$ and $K_2$ of $G$ such that:
\begin{enumerate}
\item[(1)] $\tilde{W}$ is simultaneously trivial in $K_1$ and  $K_2$ and
\item[(2)] the genus of the quotient $C_{\langle K_1, K_2 \rangle}$ is zero.
\end{enumerate}

Then $\tilde{B}_{\tilde{W}}=0.$
\end{lemm}

\begin{proof} Let us assume that $\tilde{B}_{\tilde{W}}$ has positive dimension. It is clear that there exists an abelian subvariety $P$ of $JC$ such that $$JC \sim \tilde{B}_{\tilde{W}}^{\tilde{n}} \times P,$$where $\tilde{n}=d_{\tilde{V}}/s_{\tilde{V}}$ and $\tilde{V}$ is a complex irreducible representation of $G$ associated to $\tilde{W}.$ Now, as the representation $\tilde{W}$ is simultaneously trivial in $K_1$ and  $K_2$ then it is also trivial in the group $\langle K_1, K_2\rangle;$ or, equivalently, $$\tilde{n}^{\langle K_1, K_2\rangle}=\tilde{n}.$$ It follows that the induced isogeny decomposition of $JC_{\langle K_1, K_2 \rangle}$ is $$JC_{\langle K_1, K_2 \rangle} \sim  \tilde{B}_{\tilde{W}}^{\tilde{n}} \times Q$$for a suitable abelian subvariety $Q$ of $JC.$

The result is now clear because the genus of $C_{\langle K_1, K_2\rangle}$ is assumed to be zero.
\end{proof}

\begin{theo} \label{comoKR} Let us consider a compact Riemann surface $C$ with action of a group $G.$ If there exist normal subgroups $K_1, \ldots, K_t$ of $G$ such that the genus of ${C_{\langle K_i, K_j \rangle}}$ is zero, for every $1 \le i \neq j \le t,$ then
$$JC \sim JC_{K_1} \times \cdots \times JC_{K_t} \times P$$for some abelian subvariety $P$ of $JC.$ Moreover if, in addition, $g_C=\Sigma_{i=1}^tg_{C_{K_i}}$
then $$JC \sim JC_{K_1} \times \cdots \times JC_{K_t}.$$
\end{theo}

\begin{proof}  The set of rational irreducible representations $\mathscr{R}$ of $G$ can be written as a disjoint union  $$\mathscr{R}=\bigcup_{j=1}^t \mathscr{R}_j \cup\bigcup_{\substack{i,j=1,\\  i \neq j }}^t \mathscr{R}_{ij} \cup \mathscr{C}$$where $\mathscr{R}_j$ is the subset consisting of those representations which are trivial in $K_j$ but not in $K_i$ for every $i \neq j,$ while $\mathscr{R}_{ij}$ is the subsets consisting of those representations which are simultaneously trivial in $K_i$ and $K_j,$ and $$\mathscr{C}=\mathscr{R}-(\bigcup_{j=1}^t \mathscr{R}_j \cup\bigcup_{\substack{i,j=1,\\  i \neq j }}^t \mathscr{R}_{ij}).$$

By the previous lemma, if $\tilde{W} \in \mathscr{R}_{ij}$ then $\tilde{B}_{\tilde{W}}=0$ and therefore  \begin{equation} \label{rojito} JC \sim Q_0 \times \prod_{j=1}^t  \prod_{\tilde{W} \in \mathscr{R}_j} \tilde{B}_{\tilde{W}}^{\tilde{n}_{\tilde{W}}}\end{equation}for some abelian subvariety $Q_0$ of $JC.$  Now, we note that in \eqref{rojito} $$\tilde{n}_{\tilde{W}}=\tilde{n}_{\tilde{W}}^{K_l}$$if and only if $\tilde{W} \in \mathscr{R}_l.$ Thereby, the induced isogeny decomposition of $JC_{K_l}$ is $$JC_{K_l} \sim Q_l \times \prod_{\tilde{W} \in \mathscr{R}_l} \tilde{B}_{\tilde{W}}^{\tilde{n}_{\tilde{W}}} $$for some abelian subvariety $Q_l$ of $JC,$ for each $1 \le l \le t. $ This shows that $$JC \sim JC_{K_1} \times \cdots \times JC_{K_t} \times P$$where $P:=Q_0 \times Q_1 \times \cdots \times Q_t,$ proving the first statement.

Clearly $P=0$ provided that $g_C=\Sigma_{i=1}^tg_{C_{K_i}}$ and the proof is done.
\end{proof}

\begin{rema} 	\it
In \cite{rubiyo} the authors have considered the same kind of decompositions, but on Riemann surfaces admitting actions of non necessarily normal subgroups; see also \cite{KR}.
\end{rema}

\section{Examples} \label{ejems}

\subsection{Example 1} Let $q \ge 3$ be a prime number. We consider a three-dimensional family of compact Riemann surfaces $S$ of genus $g_S=1+2q$ admitting a group of automorphism $H$ isomorphic to the dihedral group  $$\mathbf{D}_{2q}=\langle a, b : a^{2q}=b^2 = (ab)^2 =1\rangle$$of order $4q,$ acting on $S$ with signature $(0; 2, \stackrel{6}{\ldots}, 2).$ The tuple $$\tilde{\sigma} := (ba,ba,b,b,b,b)$$can be chosen as a generating vector of $H$ of the desired type.

\s

The complex irreducible representations of the dihedral groups are well-known. Namely, $\mathbf{D}_{2q}$ has four complex irreducible representation of degree one, given by $$V_1: a \mapsto 1, b \mapsto 1;   \hspace{0.4 cm} V_2: a \mapsto 1, \, b \mapsto -1;$$ $$V_3: a \mapsto -1, \, b \mapsto 1;  \hspace{0.4 cm} V_4: a \mapsto -1, \, b \mapsto -1,  $$ all of them being rationals; set $W_j=V_j.$ Additionally, it has $q-1$ complex irreducible representations of degree two:
$$
V_{k+4} : a \mapsto \mbox{diag}(\xi^k, \bar{\xi}^k),
 \,\,\,\,\, b \mapsto \left( \begin{smallmatrix}
0 & 1 \\
1 & 0
\end{smallmatrix} \right)$$where $\xi=\mbox{exp}(2 \pi i / 2q)$ and $1 \le k \le q-1.$

We also recall that the dihedral groups only possess representations with Schur index one. Hence, as the character field of each $V_l$, for $5 \le l \le q+3,$ has degree $(q-1)/2$ over $\mathbb{Q},$ they give rise to exactly two rational irreducible representations of $H,$ say $W_5$ and $W_6$ (associated to $V_5$ and $V_6$ respectively).

\s

Thus, the group algebra decomposition of $JS$ with respect to $H$ is $$JS \sim_H JS_H \times  B_{2} \times B_{3} \times B_4 \times B_5^2 \times B_6^2.$$ The dimension of the factors is summarized in the following table:

\s

\begin{center}
\begin{tabular}{|c|c|c|c|c|c|c|c|}  \hline
$\mbox{Factor}$ &  $JS_H$ & $B_2$  & $B_3$ & $B_4$  & $B_5$ & $B_6$ \\ \hline
$\mbox{Dimension}$  & 0 & 2 & 0 & 1 & $(q-1)/2$ & $(q-1)/2$ \\ \hline
\end{tabular}
\end{center}

\s

Now, in order to apply our results in two different situations, we consider the following two abstract groups of order $8q:$

\begin{enumerate}
\item the dihedral group $$\mathbf{D}_{4q}=\langle r, s : r^{4q}=s^2 = (sr)^2 =1\rangle.$$
\item the semidirect product  $$\mathbf{D}_{2q} \rtimes \mathbb{Z}_2 = \langle r, s, x : r^{2q}=s^2 = (sr)^2 =x^2 =(xr)^2=(xs)^2=1\rangle.$$
\end{enumerate}

\s

These groups act as groups of automorphisms -say $G_1$ and $G_2$ respectively- of families of compact Riemann surfaces -say $C_1$ and $C_2$ respectively- of genus $$g_1=1+8q \hspace{0.5 cm} \mbox{and}  \hspace{0.5 cm}g_2=1+4q,$$and signatures $s_1=(0; 2, \stackrel{8}{\ldots}, 2)$ and $(0; 2, \stackrel{6}{\ldots}, 2),$ respectively. The tuples $$\sigma_1=(sr,sr,s,s,s,s,r^{2q}, r^{2q}) \hspace{0.5 cm} \mbox{and}  \hspace{0.5 cm} \sigma_2 := (sr, sr, s, s, x, x)$$can be chosen as generating vectors of the desired type, in each case.

\s

The correspondences\begin{equation*} \Phi_1 : r \mapsto a, \hspace{0.3 cm} s \mapsto b \hspace{0.5 cm}  \mbox{and} \hspace{0.5 cm}  \Phi_2 : r \mapsto a, \hspace{0.3 cm} s \mapsto b, \hspace{0.3 cm} x \mapsto b\end{equation*}define epimorphisms of groups $\Phi_j : G_j \to H$ for $j=1,2;$ let $K_j$ denote the kernel of $\Phi_j.$ It is a straightforward task to check the following facts:

\begin{enumerate}
\item $\Phi_j(\sigma_j)=\sigma.$
\item $K_1=\langle r^{2q} \rangle \cong \mathbb{Z}_2$ acts with $8q$ fixed points on $C_1.$
\item $K_2=\langle xs \rangle \cong \mathbb{Z}_2$ acts freely on $C_2.$
\item $H \cong G_j/K_j$ acts on $S \cong (C_j)_{K_j}.$
\end{enumerate}

\s

We summarize the facts above in the next commutative diagram
$$\begin{tikzpicture} [node distance=1 cm, auto]
  \node (P) {$S$};
  \node (B) [right of=P, above of=P, node distance=0.9 cm] {$C_2$};
    \node (C) [left of=P, above of=P, node distance=0.9 cm] {$C_1$};
  \node (A) [below of=P, node distance=0.9 cm] {$\mathbb{P}^1$};
  \draw[->] (P) to node {{\tiny $\pi_H$}} (A);
  \draw[->] (C) to node  {{\tiny $\pi_{1}$}} (P);
    \draw[->] (B) to node [swap] {{\tiny $\pi_{2}$}} (P);
  \draw[->, bend right]  (C) to node [swap] {{\tiny$\pi_{G_1}$}} (A);
    \draw[->, bend left] (B) to node  {{\tiny $\pi_{G_2}$}} (A);
\end{tikzpicture}$$
Hence, by applying Theorem \ref{theolindo} to these two cases, we can assert that $$JC_1 \sim   {B}_2 \times  {B}_4 \times {B}_5^2 \times {B}_6^2 \times P_1$$for some abelian subvariety $P_1=P(C_1/S)$ of $JC_1$ of dimension $6q$, and $$JC_2 \sim   {B}_2  \times {B}_4 \times {B}_5^2 \times {B}_6^2 \times P_2$$for some abelian subvariety $P_2=P(C_2/S)$ of $JC_2$ of dimension $2q.$

\s

In order to further decompose $P_j$ to obtain the group algebra decomposition of $JC_j$ with respect to $G_j$ we only need to care about of those representations of $G_j$ which are not trivial in $K_j.$ Note that, by contrast, to have a complete knowledge of the irreducible representations of $G_2$ is not a trivial task.

\s

We proceed to study both cases separately.

\s

\begin{enumerate}
\item  For $j=1.$ Let us consider the two-dimensional complex irreducible representations of $G_1$ given by $$
r \mapsto \mbox{diag} (i, -i),\,\,\,\,\, s \mapsto \left( \begin{smallmatrix}
0 & 1 \\
1 & 0
\end{smallmatrix} \right) \hspace{0.3 cm}\mbox{and}\hspace{0.3 cm}
r \mapsto \mbox{diag} (\omega, \bar{\omega}), \,\,\,\,\, s \mapsto \left( \begin{smallmatrix}
0 & 1 \\
1 & 0
\end{smallmatrix} \right)$$where $\omega=\mbox{exp}(2 \pi i / 4q).$

It is easy to see that both of them are not trivial in $K_1$ and that the character fields are $\mathbb{Q}$ and $\mathbb{Q}(\omega + \bar{\omega})$ respectively. Clearly, they are not Galois associated and, in consequence, yield two rational irreducible representations of $G_1$ --say $\tilde{W}_7$ and $\tilde{W}_8$--  and hence two factors in the group algebra decomposition of $P_1$ with respect to $G_1.$ Thus $$P_1 \sim \tilde{B}_7^2 \times  \tilde{B}_8^2 \times Q_1$$for some abelian subvariety $Q_1$ of $P_1.$ Applying the formula \eqref{dimensiones}, we see that $$\mbox{dim}(\tilde{B}_7)=3 \hspace{0.5cm} \mbox{and} \hspace{0.5cm} \mbox{dim} (\tilde{B}_8)=3(q-1)$$and, consequently, $Q_1=0.$ Thereby, the following isogeny is obtained
$$JC_1 \sim_{G_1} ({B}_2 \times {B}_4 \times {B}_5^2 \times {B}_6^2 ) \times (\tilde{B}_7^2 \times  \tilde{B}_8^2).$$

\s

\item  For $j=2.$ Let us consider the one-dimensional complex irreducible representations of $G_2$ given by
\begin{equation*}
r \mapsto -1, \,\,\,\,\, s \mapsto 1, \,\,\,\,\, x \mapsto -1 \hspace{0.4 cm} \mbox{and}\hspace{0.4 cm} r \mapsto 1, \,\,\,\,\, s \mapsto -1, \,\,\,\,\, x \mapsto 1,
\end{equation*}and the two-dimensional ones given by
$$r \mapsto \mbox{diag} (\mu,  \bar{\mu} ) \,\,\,\,\, s \mapsto \left( \begin{smallmatrix}
0 & -1 \\
-1 & 0
\end{smallmatrix} \right) \,\,\,\,\, x \mapsto \left( \begin{smallmatrix}
0 & 1 \\
1 & 0
\end{smallmatrix} \right)$$and$$
r \mapsto\mbox{diag} (\mu^2,  \bar{\mu}^2 ) \,\,\,\,\, s \mapsto \left( \begin{smallmatrix}
0 & 1 \\
1 & 0
\end{smallmatrix} \right) \,\,\,\,\, x \mapsto \left( \begin{smallmatrix}
0 & -1 \\
-1 & 0
\end{smallmatrix} \right)$$where $\mu=\mbox{exp}(2 \pi i / q).$

It is easy to see that each one of them is not trivial in $K_2$ and that the character fields are $\mathbb{Q}$ and $\mathbb{Q}(\mu + \bar{\mu})$ respectively. Clearly, they are not Galois associated and, in consequence, yield four rational irreducible representations of $G_2$ --say $\hat{W}_7, \hat{W}_{8}, \hat{W}_9, \hat{W}_{10}$--  and hence four factors in the group algebra decomposition of $P_2$ with respect to $G_2.$ Thus $$P_2 \sim \hat{B}_7 \times \hat{B}_8 \times \hat{B}_9^2 \times \hat{B}_{10}^2 \times Q_2$$for some abelian subvariety $Q_2$ of $P_2.$ Applying the formula \eqref{dimensiones}, we see that  $$\mbox{dim}(\hat{B}_7)=\mbox{dim}(\hat{B}_8)=1 \hspace{0.3cm} \mbox{and} \hspace{0.3cm} \mbox{dim} (\hat{B}_9)=\mbox{dim} (\hat{B}_{10})=(q-1)/2$$and, consequently, $Q_2=0.$ Thereby, the following isogeny is obtained$$JC_2 \sim_{G_2} ({B}_2 \times {B}_4 \times {B}_5^2 \times {B}_6^2 ) \times \hat{B}_7 \times \hat{B}_8 \times  \hat{B}_9^2 \times  \hat{B}_{10}^2.$$

\end{enumerate}
\s

Finally, let us consider the subgroup $N=\langle a^q, b \rangle \cong \mathbb{Z}_2^2$ of $H.$ Note that \begin{displaymath}
d_{{V}_j}^N= \left\{ \begin{array}{ll}
0 & \textrm{if $j  = 2,4,5$}\\
1 & \textrm{if $j=6$}
  \end{array} \right.
\end{displaymath} ($B_1=B_3=0$) and therefore $B_{6} \sim JS_{N}.$ By Theorem \ref{corolindo}, it follows that $\tilde{B}_6 := \pi_{1}^*(B_6) \sim J{C_1}_{N_1}$ where $$N_1=\Phi_1^{-1}(N)= \langle r^q, s \rangle \cong \mathbf{D}_4,$$while $\hat{B}_6:= \pi_{2}^*(B_6) \sim J{C_2}_{N_2}$ where $$N_2=\Phi_2^{-1}(N)= \langle r^q, s, x \rangle \cong \mathbb{Z}_2^3.$$

\s

\subsection{Example 2} Let $q \ge 5$ be a prime number and let $C$ be a compact Riemann surface with action of  $G \cong \mathbb{Z}_q^2$ such that the signature of the quotient is $(0; q,q,q).$ If we denote by $a_1$ and $a_2$ two generators of $G$ acting with fixed points, then$$\sigma=(a_1, a_2, (a_1a_2)^{-1})$$can be chosen as a generating vector of $G$ of the desired type. The subgroups $$K_i=\langle a_1 a_2^{i} \rangle \cong \mathbb{Z}_q \hspace{0.5 cm} 2 \le i \le q-1$$act on $C$ and $\langle K_{i}, K_j \rangle=G$ for $2 \le i \neq j \le q-1.$ Now, by Theorem \ref{comoKR}, the isogeny$$JC \sim_G JC_{K_2} \times \cdots \times JC_{K_{q-1}} \times P$$ is obtained, for a suitable abelian subvariety $P$ of $JC.$ Furthermore, as the genus of $C_{K_i}$ equals $(q-1)/2$ for all $2 \le i \le q-1$ we have that $$ g_{K_2} + \cdots + g_{K_{q-1}}=(q-1)(q-2)/2=g_C$$and, consequently, $P=0.$ We obtain the isogeny $$JC \sim JC_{K_2} \times \cdots \times JC_{K_{q-1}}.$$

\end{document}